\documentclass[12pt,oneside]{amsart}
\RequirePackage[colorlinks,citecolor=blue,urlcolor=blue]{hyperref}

\usepackage{latexsym,amssymb,amsmath,amsfonts,amsthm}
\usepackage{amsthm,amsmath}
\usepackage{float}
\usepackage{enumerate}
\usepackage{epic}
\usepackage{fullpage}
\usepackage{bbm}
\usepackage{subfigure}
\usepackage{graphicx}
\usepackage[toc,page]{appendix}
\usepackage{amsfonts}
\usepackage[all]{xy}
\usepackage{caption}
 \usepackage{geometry}              
\usepackage{graphicx}
\usepackage{amssymb}
\usepackage{epstopdf}
\usepackage{color}
\usepackage{bbm, dsfont}
\usepackage{hyperref}
\usepackage[utf8]{inputenc}
\usepackage{amssymb,amsthm,amsmath,amssymb,wrapfig,dsfont}
\usepackage[dvipsnames]{xcolor}
\definecolor{myred}{RGB}{251,154,133}
\definecolor{myblue}{RGB}{153,206,227}
\definecolor{mylightblue}{RGB}{0, 150, 255}
\definecolor{mygreen}{RGB}{32, 210, 64}
\definecolor{mygray}{RGB}{220, 220, 220}

\usepackage{tikz}
\usetikzlibrary{decorations.pathmorphing}
\tikzset{snake it/.style={decorate, decoration=snake}}
\usetikzlibrary{shapes.geometric,positioning,decorations.pathreplacing}

\newtheorem{theorem}{Theorem}

\newtheorem{lemma}{Lemma}[section]
\newtheorem{remark}{Remark}
\newtheorem{Proposition}{Proposition}

\DeclareFontFamily{OML}{rsfs}{\skewchar\font'177}
\DeclareFontShape{OML}{rsfs}{m}{n}{ <5> <6> rsfs5 <7> <8> <9>
rsfs7 <10> <10.95> <12> <14.4> <17.28> <20.74> <24.88> rsfs10 }{}
\DeclareMathAlphabet{\mathfs}{OML}{rsfs}{m}{n}

\newcommand{\BH}{{\mathbb{H}}}
\newcommand{\BI}{{\mathbb{I}}}

\newcommand{\BN}{{\mathbb{N}}}

\newcommand{\BZ}{{\mathbb{Z}}}

\newcommand{\CH}{{\mathcal{H}}}

\newcommand{\ind}{{\mathbbm{1}}}

\newcommand{\prob}{{\bf P}}

\newcommand{\bae}{\begin{equation}\begin{aligned}}
\newcommand{\eae}{\end{aligned}\end{equation}}
\newcommand{\ev}{\mathbf{E}}

\newcommand{\diam}{\text{diam}}

\newcommand{\harm}{\CH}

\def\beq{ \begin{equation} }
\def\eeq{ \end{equation} }

\def\square{\vcenter{\vbox{\hrule height .4pt
  \hbox{\vrule width .4pt height 5pt \kern 5pt
        \vrule width .4pt} \hrule height .4pt}}}

\begin{document}

\title{Stationary DLA is well defined}

\author{Eviatar B. Procaccia}
\address[Eviatar B. Procaccia\footnote{Research supported by NSF grant DMS-1812009}]{Texas A\&M University}
\urladdr{www.math.tamu.edu/~procaccia}
\email{eviatarp@gmail.com}

\author{Jiayan Ye}
\address[Jiayan Ye]{Texas A\&M University}
\urladdr{http://www.math.tamu.edu/~tomye}
\email{tomye@math.tamu.edu }
 
\author{Yuan Zhang}
\address[Yuan Zhang]
{Peking University}
\email{zhangyuan@math.pku.edu.cn}

\maketitle

\begin{abstract}
In this paper, we construct an infinite stationary Diffusion Limited Aggregation (SDLA) on the upper half planar lattice, growing from an infinite line, with local growth rate proportional to the stationary harmonic measure. We prove that the SDLA is ergodic with respect to integer left-right translations.
\end{abstract}


\section{Introduction}

Diffusion limited aggregation (DLA) is a set valued process first defined by Witten and Sander \cite{DLA_introduction} in order to study physical systems governed by diffusion. DLA is defined recursively as a process on subsets of $\BZ^2$. Starting from $A_0=\{(0,0)\}$, at each time a new point $a_{n+1}$ sampled from the harmonic probability measure on the outer vertex boundary of $A_n$ is added to $A_n$. Intuitively, $a_{n+1}$ is the first place that a random walk starting from infinity visits $\partial^{out}A_n$.

In many experiments and real world phenomenon the aggregation grows from some initial boundary instead of a single point i.e. ions diffusing in liquid until they connect a charged container floor (see \cite{ball1986dla} for numerous examples). Different aggregation processes, such as Eden and Internal DLA, with boundaries were studied in \cite{MR3619794,Berger-Kagan-Procaccia}, and universal phenomenon such as a.s. non existence of infinite trees were proved. 

 In this paper we construct an infinite stationary DLA (SDLA) on the upper half planar lattice, growing from an infinite line. Along the way we prove that this infinite stationary DLA can be seen as a limit of DLA growing from a long finite line. This Allows one to use the more symmetric and amenable model of SDLA to study local behavior of DLA. In addition SDLA admits new phenomena not observed in the full lattice DLA. One such interesting conjectured phenomenon, which results from the competition between different trees in the SDLA, is that eventually (and in finite time) every tree in the SDLA ceases to grow.

\section{Statement of Result}

The main result we obtained in the paper is the well-definition of the (infinite) SDLA according to its transition rate given by the stationary harmonic measure, starting from the infinite initial configuration $L_0$. 
\begin{theorem}
\label{theorem_well_define}
Let $t>0$ and $A_0=L_0$, then there is a well defined SDLA process $\{A^\infty_s\}_{s\le t}$.
\end{theorem}

\begin{remark}
The result remains true if one replace the initial state $L_0$ by any subset  $A_0$ that can be seen as a connected forest of logarithmic horizontal growth rate. To be precise, $A_0$ can be written as $\cup_{n=-\infty}^\infty Tree^n_0$, where $Tree^n_0$ is connected for each $n$, with $Tree^n_0\cap L_0=(n,0)$ and moreover $\diam(Tree^n_0)\ge \log n$ for only finite number of $n$'s. We present the proof for $A_0=L_0$ for simplicity but without loss of (much) generality. 
\end{remark}


A major tool one obtains for the study of SDLA is ergodicity of the process.
\begin{theorem}\label{thm:ergodicity}
For every $t>0$, $A^\infty_t$ is ergodic with respect to shift in $\BZ\times\{0\}$.
\end{theorem}

\section{Preliminaries}
We first recall a number of notations and results from a previous paper by two of the authors \cite{Procaccia_upper_2019}: Let $\BH=\{(x,y)\in \mathbb{Z}^2, y\ge 0\}$ be the upper half plane (including $x$-axis), and $S_n, n\ge 0$ be a 2-dimensional simple random walk. For any $x\in \mathbb{Z}^2$, we will write 
$$
x=(x_1,x_2)
$$
with $x_i$ denote the $i$th coordinate of $x$, and $\|x\|=\|x\|_1=|x_1|+|x_2|$. Then let $L_n, D_n\subset \mathbb{Z}^2$ be defined as follows: for each nonnegative integer $n$, define
$$
L_n=\{(x,n), \ x\in \mathbb{Z}\}
$$
to be the horizontal line of height $n$. For each subset $A\subset \mathbb{Z}^2$ we define the stopping times 
$$
\bar \tau_A=\min\{n\ge 0, \ S_n\in A \}
$$
and
$$
\tau_A=\min\{n\ge 1, \ S_n\in A \}.
$$

For any subsets $A_1\subset A_2$ and $B$ and any $y\in \mathbb{Z}^2$, by definition one can easily check that 
\begin{equation} 
\label{basic 1}
\begin{aligned}
&\prob_y\left(\tau_{A_1}<\tau_{B}\right)\le \prob_y\left(\tau_{A_2}<\tau_{B}\right), \\
&\prob_y\left(\bar\tau_{A_1}<\bar\tau_{B}\right)\le \prob_y\left(\bar\tau_{A_2}<\bar\tau_{B}\right), 
\end{aligned}
\end{equation} 
and that 
\begin{equation} 
\label{basic 2}
\begin{aligned}
\prob_y\left(\tau_{B}<\tau_{A_2}\right)\le \prob_y\left(\tau_{B}<\tau_{A_1}\right),\\
\prob_y\left(\bar\tau_{B}<\bar\tau_{A_2}\right)\le \prob_y\left(\bar\tau_{B}<\bar\tau_{A_1}\right),
\end{aligned}
\end{equation} 
where $\prob_y(\cdot)=\prob(\cdot|S_0=y)$. Then in \cite{Procaccia_upper_2019} we defined the stationary harmonic measure on $\BH$ which will serve as the Poisson intensity in our continuous time DLA model. For any $B\subset \BH$, any edge $\vec e=x\to y$ with  $x\in B$, $y\in \BH\setminus B$ and any $N$, we define 
\begin{equation} 
\label{harmonic measure edge}
\harm_{B, N}(\vec e)=\sum_{z\in L_N\setminus B} \prob_z\left(S_{\bar \tau_{B\cup L_0}}=x, S_{\bar \tau_{B\cup L_0}-1}=y\right)
\end{equation} 
By definition, a necessary condition for $\harm_{B, N}(\vec e)>0$ is $y\in \partial^{out}B$ and $|x-y|=1$. And for all $x\in B$, we can also define  
\begin{equation} 
\label{harmonic measure point old}
\harm_{B, N}(x)=\sum_{y: \ \vec e=(x,y)}\harm_{B, N}(\vec e) =\sum_{z\in L_N\setminus B} \prob_z\left(S_{\bar \tau_{B\cup L_0}}=x\right).
\end{equation} 
And for each point $y\in \partial^{out}B$, we can also define 
\begin{equation} 
\label{harmonic measure point new}
\hat \harm_{B, N}(y)=\sum_{\vec e=(x,y), ~x\in B}\harm_{B, N}(\vec e) =\sum_{z\in L_N\setminus B} \prob_z\left(\tau_{B}\le \tau_{L_0} , S_{\bar \tau_{B\cup L_0}-1}=y\right).
\end{equation} 
By coupling and strong Markov property, we show that $N\to \harm_{A,N}(e)$ is bounded and monotone in $N$. Thus we proved that 
\begin{Proposition}[Proposition 1, \cite{Procaccia_upper_2019}]
\label{proposition_well_define}
For any $B$ and $\vec e$ as above, there is a finite $\harm_{B}(\vec e)$ such that 
\begin{equation} 
\lim_{N\to\infty} \harm_{B, N}(\vec e)=\harm_{B}(\vec e). 
\end{equation} 
\end{Proposition}
And $\harm_{B}(\vec e)$ is called the stationary harmonic measure of $\vec e$ with respect to $B$. The following limits $\harm_{B}(x)=\lim_{N\to\infty}\harm_{B, N}(x)$ and $\hat \harm_{B}(y)=\lim_{N\to\infty}\hat \harm_{B, N}(y)$ also exist \cite{Procaccia_upper_2019} and are called the stationary harmonic measure of $x$ and $y$ with respect to $B$.  

Then for any connected $B\subset\BH$ such that $B\cap L_0\not=\emptyset$, and any $x\in B$, $\harm_B(x)$ was proved to have the following up bounds that depends only on the height of $x$:
\begin{theorem}[Theorem 1, \cite{Procaccia_upper_2019}]
\label{theorem: uniform_path}
There is some constant $C<\infty$ such that for each connected $B\subset \BH$ with $L_0\subset B$ and each $x=(x_1,x_2)\in B\setminus L_0$, and any $N$ sufficiently larger than $x_2$
\begin{equation} 
\label{uniform bound}
\harm_{B, N}(x)\le C x_2^{1/2}. 
\end{equation} 
\end{theorem}
\begin{remark}
It is easy to note that for any $B\subset\BH$ such that $B\cap L_0\not=\emptyset$ and any $x=(x_1,x_2)\in B\setminus L_0$, $\harm_B(x)=\harm_{B\cup L_0}(x)$. Thus one may without loss of generality assume that $L_0\subset B$. 
\end{remark}

\begin{remark}
Since the constant $C$ above does not depend on subset $B$ or point $x$, without loss of generality, one may (incorrectly) assume $C=1$.
\end{remark}

With the upper bounds of the harmonic measure on the upper half plane, a pure growth model called the {\bf interface process} was introduced in \cite{Procaccia_upper_2019} which can be used as a dominating process for both the DLA model in $\BH$ and the stationary DLA model that will be introduced in this paper. Consider an interacting particle system $\bar \xi_t$ defined on $\{0,1\}^\BH$, with 1 standing for an occupied site while 0 for a vacant site, with transition rates as follows: 
\begin{enumerate}[(i)]
\item For each occupied site $x=(x_1,x_2)\in \BH$, if $x_2>0$ it will try to give birth to each of its nearest neighbors at a Poisson rate of $\sqrt{x_2}$. If $x_2=0$, it will try to give birth to each of its nearest neighbors at a Poisson rate of $1$. 
\item If $x$ attempts to give birth to a nearest neighbors $y$ that is already occupied, the birth is suppressed.   
\end{enumerate}
We proved that an interacting particle system determined by the dynamic above is well-defined. 
\begin{Proposition}[Proposition 3, \cite{Procaccia_upper_2019} ]
\label{lemma_well_definition_1}
The interacting particle system $\bar \xi_t\in \{0,1\}^\BH$ satisfying (i) and (ii) is well defined.  
\end{Proposition}
Then when the initial aggregation $V_0$ is the origin or finite, we defined the DLA process in $\BH$ starting from $V_0$ (Theorem 5, \cite{Procaccia_upper_2019}), according to the graphic representation (see \cite{ten_lectures} for introduction) of the interface process $\bar \xi_t$ and a procedure of Poisson thinning, see Page 30-31 of \cite{Procaccia_upper_2019} for details. Note that under this construction, the DLA model with finite initial aggregation keeps staying below the interface process.

\section{Coupling construction}

Now in order to prove Theorem \ref{theorem_well_define}, we constructed a sequence of processes $\{A^n_t\}_{n=1}^\infty$, each of which are the DLA in $\BH$ with initial aggregation $V^n_0=[-n,n]\times 0$, coupled together with a same interface process. To be precise, recall the graphic representation in \cite{Procaccia_upper_2019}: 
\begin{itemize}
\item For each $x=(x_1,x_2)$ and $y=(y_1,y_2)\in \BH$ such that $\|x-y\|=1$, we associate the edge $\vec e=x\to y$ with an independent Poisson process $N_t^{x\to y}, t\ge 0$ with intensity $\lambda_{x\to y}=\sqrt{x_2}\vee 1$. 
\item For each $x=(x_1,x_2)$ and $y=(y_1,y_2)\in \BH$ such that $\|x-y\|=1$ let $\{U^{x\to y}_i\}_{i=1}^\infty$ be i.i.d. sequences of $U(0,1)$ random variables independent to each other and to the Poisson processes.  
\end{itemize}
At any time $t$ when there is Poisson transition for edge  $\vec e=x\to y$, we draw the directed edge $(\vec e, t)$ in the phase spcae $\BH\times [0,\infty)$. For any $x\in L_0$ and any fixed time $t$, recall that $I^x_t$ is a subset of all $y$'s in $\BH$ which are connected with $x$ by a path going upwards vertically or following the directed edges. Then in \cite{Procaccia_upper_2019} it has been proved that for all $V_0\subset \BH$, 
$$
\bar \xi^{V_0}_t=\bigcup_{x\in V_0}I^x_t
$$
distributed as the interface process with initial state $V_0$. Moreover, it was proven that for each $t<\infty$ and all $x\in \BH$, $|I^x_t|<\infty$ with probability one, and there can be only a finite number of different paths emanating from $x$ by time $t$, which may only have finite transitions involved. Now for all finite $V_0$, in \cite{Procaccia_upper_2019} we look at the finite set of all the transitions involved in the evolution of $\bar \xi^{V_0}_s, \ s\in [0,t]$, and order them according to the time of occurrence. Then the following thinning was applied in order to define a process $A_t=(V_t, E_t)$ starting at $A_0=(V_0,\emptyset)$: when a new transition arrives at time $t_i$, say it is the $j$th Poisson transition on edge $\vec e=x\to y$. Suppose one already knew $A_{t_i-}:=\lim_{s\uparrow t_i}A_s$.  
\begin{itemize}
\item If $x\notin V_{t_i-}$ or $y\in V_{t_i-}$, nothing happens. 
\item Otherwise:
\begin{itemize}
\item If $U^{x\to y}_j \le \harm_{V_{t_i-}}(\vec e)/\lambda_{\vec e}$, then $V_{t_i}=V_{t_i-}\cup \{y\}$, $E_t=E_{t-}\cup \{\vec e\}$.
\item Otherwise, nothing happens.
\end{itemize}
\end{itemize}
Thus we defined the process $A_t$ up to all time $t$ with $V_t$ identically distributed as our DLA process starting from $A_0$. Now, for each $n$ define $A^n_t$ as the process with $A^n_0=([-n,n]\times 0,\emptyset)$. Then we have coupled all $A^n_t$'s using the same graphic representation and thinning factors. Now in order to prove Theorem \ref{theorem_well_define}, we first show the following theorem which states that for a finite space-times box, the discrepancy probabilities for our $A^n$'s are summable. 


\begin{theorem}
\label{thm_summable}
For any compact subset $K\subset \BH$ and any $T<\infty$, we have 
\beq
\label{eq:summable}
\sum_{n=1}^\infty \prob\left(\exists t\le T, \ s.t. \ A^{n}_t\cap K\not=A^{n+1}_t\cap K \right)<\infty. 
\eeq
\end{theorem}

\begin{remark}
	We will, without loss of generality in the rest of this paper assume that $T=1$.  
\end{remark}

The proof of Theorem \ref{thm_summable} is immediate once one proves that there exist constants $\alpha>0$ and $C<\infty$ such that for all sufficiently large $n$
\beq
\label{sup_linear_decay}
\prob\left(\exists t\le 1, \ s.t. \ A^{n}_t\cap K\not=A^{n+1}_t\cap K \right)\le \frac{C}{n^{1+\alpha}}. 
\eeq
Note that at $t=0$, the initial aggregations $A^{n}_0$ and $A^{n+1}_0$ are different only by the two end points $(\pm (n+1),0)$. Now we want to control the subset of the discrepancies so that they will not reach $K$ by time $1$. Intuitively, the idea we will follow in the detailed proof in the following sections can be summarized as the follows: 
\begin{enumerate}[(I)]
\item With very high probability none of $A^{n}_1$ and $A^{n+1}_1$ can reach height $\log(n)$. 
\item For any $\alpha>0$, with very high probability the two processes will not have as many as $n^\alpha$ discrepancies by time $1$.
\item For all these discrepancies ever created till time $1$, with very high probability none of them will ever find its way to $K$.  
\end{enumerate}


\section{Logarithm growth of the interface process}

In this section, we prove the logarithm growth upper bound for $A^{n}_t$ and $A^{n+1}_t$ with $t\in [0,1]$. Note that both are contained in the interface process $I^{[-n-1,n+1]\times 0}_t$. Thus it suffices to show that 
\begin{theorem}
\label{thm_log_growth}
for any $C<\infty$, 
$$
\prob\left( I^{[-n,n]\times 0}_1\nsubseteq [-n-\log n, n+\log n]\times[0,\log n] \right)<\frac{1}{n^C}
$$
for all sufficiently large $n$. 
\end{theorem}
\begin{proof}
First noting that 
$$
I^{[-n,n]\times 0}_1= \bigcup_{x\in [-n,n]\times 0}I^{x}_1,
$$
which, combining with additivity implies it suffices to show that for any $C<\infty$ and all sufficiently large $k$, 
\beq
\label{exp_decay_radius}
\prob\left( \|I^{0}_1\|_2\ge k\right)<\exp(-Ck),
\eeq
where 
$$
 \|A\|_2=\max_{x\in A}\|x\|_2
$$
for all finite $A\subset \BH$. In order to get \eqref{exp_decay_radius}, one first proves 
\begin{lemma}
\label{lem_path}
Let $\{T_i\}_{i=1}^k$ be independent exponential random variables with parameters $\lambda_i=4\sqrt{i+1}$. For any $t>0$, $\prob(\|I^0_1\|_2>k)\le 4^k\prob(\sum_{i=1}^k T_i<t)$.
\end{lemma}
\begin{proof}
	
Under the event $\{\|I^0_1\|_2>k\}$, by definition and the fact that $I^0_1$ is a nearest neighbor growth model, there has to exist a nearest neighbor sequence of points $0=x_0, x_1,\cdots, x_m$ with $\|x_m\|\ge k$ such that for stopping times 
$$
\tau_i=\inf\{s\ge 0 : \ x_i\in I^0_s\}
$$
we have that
$$
0=\tau_0<\tau_1<\cdots<\tau_m<1. 
$$
Noting that $x_0, x_1,\cdots, x_m$ is a nearest neighbor path with $\|x_m\|\ge k$, which implies $m\ge k$, we may without loss of generality assume $m=k$. More precisely, there exists a nearest neighbor sequence of points $0=x_0, x_1,\cdots, x_k$ such that for stopping times 
$$
\tau_i=\inf\{s\ge 0 : \ x_i\in I^0_s\}
$$
we have that
$$
0=\tau_0<\tau_1<\cdots<\tau_k<1.  
$$
Note that there are no more than $4^k$ such different nearest neighbor sequences of points within $\BH$ starting at $0$. And for each given path $0=x_0, x_1,\cdots, x_k$, and each $1\le i\le k$, define 
$$
\Delta_i=\min_{y: \|y-x_i\|=1} \inf\left\{s>0: \ N^{y\to x_i}_{\tau_{i-1}+s}=N^{y\to x_i}_{\tau_{i-1}}+1\right\}.
$$
Then by definition and the strong Markov property, $\Delta_i$ is an exponential random variable with rate $\hat \lambda_i=\sum_{y: \|y-x_i\|=1} \lambda_{y\to x_i}\le 4\sqrt{i+1}$, independent to $\mathcal{F}_{\tau_{i-1}}$. At the same time, note that by definition $\Delta_i\le \tau_i-\tau_{i-1}$, which implies that $\Delta_i\in \mathcal{F}_{\tau_{i}}$, and that $\{\Delta_i\}_{i=1}^k$ is a sequence of independent random variables. Thus
$$
\prob(\tau_0<\tau_1<\cdots<\tau_k<1)\le \prob\left(\sum_{i=1}^k \Delta_i<1 \right)\le  \prob\left(\sum_{i=1}^k T_i<1 \right).
$$
\end{proof}

For some constants $c_1, c_2>0$ (to be chosen later) define the event 
$$G=\left\{\left|\left\{1\le i\le k: T_i\ge \frac{c_2}{\sqrt{i}}\right\}\right|>c_1 k\right\}.$$
\begin{lemma}
For any $t>0$ and $k\in\BN$ large enough, $\prob(\sum_{i=1}^k T_i<t)\le \prob(G(t)^c)$.
\end{lemma}
\begin{proof}
Under the event $G$, 
\begin{align}
\sum_{i=1}^k T_i\ge \sum_{i: ~T_i\ge \frac{c_2}{\sqrt{i}}} T_i\ge c_1 k \frac{c_2}{\sqrt{k}}=c_1c_2\sqrt{k}\ge 1,
\end{align}
where the last inequality holds for any sufficiently large $k$.
\end{proof}
\begin{lemma}
\label{lem_large_dev}
Let $t>0$ any $\tilde c\in (0,\infty)$, then there exists $c_1,c_2>0$ such that for any sufficiently large $k$,
$$
\prob(G^c)\le\exp(-\tilde c k)
.$$
\end{lemma}
\begin{proof}
Define $X_i=\ind_{\left\{T_i\ge \frac{c_2}{\sqrt{i}}\right\}}$, thus $\sum_{i=1}^k X_i$ is a binomial random variable with parameters $n$ and $p=\prob\left(T_i\ge \frac{c_2}{\sqrt{i}}\right)=e^{-c_2}$, which converges to $1$ when $c_2\rightarrow 0$. By the large deviation principle for the binomial distribution
$$\prob\left(\sum_{i=1}^k T_i< c_1 k\right) \le e^{-I(c_1,p)k}.$$ For $p$ close enough to $1$ we have $I(c_1,p)>\tilde c$ (see \cite{MR2571413} for the exact rate function).
\end{proof}
\noindent {\it Proof of Theorem \ref{thm_log_growth}.} For any $C\in (0,\infty)$, fix a $\tilde c= C+\log(4)+1$. Then Theorem \ref{thm_log_growth} follows from the combination of \eqref{exp_decay_radius} and Lemma \ref{lem_path}-\ref{lem_large_dev}.
\end{proof}

\section{Truncated processes and number of discrepancies}

In the section we complete Step (II) in the outline. But prior to that, we would like to use Theorem \ref{thm_log_growth} to define a truncated version of coupled process $(A^n_t, A^{n+1}_t)$. Define stopping time 
$$
\Gamma=\inf\left\{t\ge 0: \ V^n_t\cup V^{n+1}_t\nsubseteq [-n-\log n, n+\log n]\times[0,\log n] \right\}
$$
be the first time $A^n_t$ or $A^{n+1}_t$ grows outsides the box $[-n-\log n, n+\log n]\times[0,\log n]$. 
\begin{remark}
	It is easy to see that $V^n_t$ or $V^{n+1}_t$ grows outsides our box if and only if $E^n_t$ or $E^{n+1}_t$ does so. 
\end{remark}
Now we can define the {\bf truncated processes} 
$$
(\hat A^n_t, \hat A^{n+1}_t)=\left(A^n_{t\wedge \Gamma}, A^{n+1}_{t\wedge \Gamma}\right).
$$
I.e., we have the coupled processes stopped once either of them goes outsides the box $[-n-\log n, n+\log n]\times[0,\log n]$. By definition, we have 
$$
(A^n_t, A^{n+1}_t)=(\hat A^n_t, \hat A^{n+1}_t)
$$
for all $t\in [0,\Gamma]$. At the same time, note that 
$$
V^n_t\cup V^{n+1}_t\subset \bigcup_{x\in [-n-1,n+1]\times 0}I^{x}_t
$$
for all $t\ge 0$. Thus for all $C<\infty$ and all sufficiently large $n$, 
\beq
\label{eq_truncated_couple}
\begin{aligned}
\prob&\left(A^n_t\equiv \hat A^n_t, A^{n+1}_t\equiv \hat A^{n+1}_t, \ \forall t\in [0,1] \right)\\
&\le \prob\left(I^{[-n-1,n+1]\times 0}_1\nsubseteq [-n-\log n, n+\log n]\times[0,\log n] \right)<\frac{1}{n^C}.
\end{aligned}
\eeq
Thus in order to show Theorem \ref{thm_summable}, it suffices to prove that there exists constants $\alpha>0$ and $C<\infty$ such that for all sufficiently large $n$
\beq
\label{truncate_sup_linear_decay}
\prob\left(\exists t\le1, \ s.t. \ \hat A^{n}_t\cap K\not=\hat A^{n+1}_t\cap K \right)\le \frac{C}{n^{1+\alpha}}. 
\eeq
Now we formally define the set of discrepancies for the coupled process $(\hat A^n_t, \hat A^{n+1}_t)$. For any $t<\infty$, define 
$$
V_t^{D,n}=\left\{x\in\BH, \ s.t.\ \exists s\le t, \ x\in \hat V^n_s\triangle \hat V^{n+1}_s \right\}
$$
as the set of {\bf vertex discrepancies}, and
$$
E_t^{D,n}=\left\{\vec e=x\to y, \ x,y\in \BH, \ s.t.\ \exists s\le t, \ \vec e\in \hat E^n_s\triangle \hat E^{n+1}_s \right\}
$$
as the set of {\bf edge discrepancies}, where $\triangle$ stands for the symmetric difference of sets. From their definition, we list some basic properties of the sets of discrepancies as follows: 
\begin{itemize}
	\item Both $V_t^{D,n}$ and $E_t^{D,n}$ are non-decreasing with respect to time. 
	\item For any $x\in V_t^{D,n}$, there has to be an edge $\vec e_x\in E_t^{D,n}$ ending at $x$.  
	\item For any $\vec e= a\to x\in E_t^{D,n}$, $x$ has to be in $x\in V_t^{D,n}$. 
	\item Whenever a new vertex is added in $V_t^{D,n}$, there has to be a new edge added to $E_t^{D,n}$. However, when a new edge is added to $E_t^{D,n}$, there may or may not be a  a new vertex added in $V_t^{D,n}$.
\end{itemize}
From the observations above, it is immediate to see that $V_t^{D,n}$ is the same as the collection of all ending points in $E_t^{D,n}$, which also implies that $|V_t^{D,n}|\le |E_t^{D,n}|$.

Moreover, for the event of interest, we have 
\beq
\label{discrepancy_event}
\left\{\exists t\le1, \ s.t. \ \hat A^{n}_t\cap K\not=\hat A^{n+1}_t\cap K \right\}=\left\{V_1^{D,n}\cap K\not=\emptyset\right\}.
\eeq
As we outlines in the previous section, in order to prove the event in \eqref{discrepancy_event} has a super-linearly decaying probability as $n\to\infty$, we first control the growth of $|E^{D,n}_t|$. I.e., by time 1 there cannot be too many discrepancies created in the coupled system. To be precise, we prove that 
\begin{lemma}
	\label{lemma_number_discrepancies}
	For any $\alpha>0$, there is a $c>0$ such that 
	$$
	\prob\left(|E^{D,n}_1|\ge n^\alpha \right)\le \exp(-n^c)
	$$
	for all sufficiently large $n$. 
\end{lemma}
\begin{proof}
Note that $|E^{D,n}_0|=0$. For $i=1,2,\cdots$, define stopping time $\Delta_i=\inf\{t\ge 0, \ |E^{D,n}_t|=i\}$, with the convention $\inf\emptyset=\infty$. Given the configuration of $(\hat A^{n}_t,\hat A^{n+1}_t)$, we first discuss the rate at which a new discrepancy is created. If $t\le \Gamma$, by definition such rate equals to 0. Otherwise, for each edge $\vec e= x\to y$ in $\BH$, it can be classified according to the configuration as follows: define indicator matrix
$$
\BI(\hat A^{n}_t,\hat A^{n+1}_t)(\vec e)=\left(
\begin{array}{lll}
\ind_{x\in \hat V^{n}_t} & \ind_{y\in \hat V^{n}_t} & \ind_{\vec e\in \hat E^{n}_t}\\
\ind_{x\in \hat V^{n+1}_t} & \ind_{y\in \hat V^{n+1}_t} & \ind_{\vec e\in \hat E^{n+1}_t}
\end{array}
\right).
$$
Then by definition, the only edges that contribute to the increasing rate of $E_t^{D,n}$ are those with indicator matrices as one of the followings:
\begin{align*}
&\BI_1=
\left(
\begin{array}{lll}
1 & 0 & 0\\
1 & 0 & 0
\end{array}
\right), \
\BI_2=
\left(
\begin{array}{lll}
1 & 1 & 0\\
1 & 0 & 0
\end{array}
\right),\\
&\BI_3=
\left(
\begin{array}{lll}
1 & 0 & 0\\
0 & 0 & 0
\end{array}
\right), \
\BI_4=
\left(
\begin{array}{lll}
1 & 0 & 0\\
0 & 1 & 0
\end{array}
\right),\\
&\BI_5=
\left(
\begin{array}{lll}
1 & 0 & 0\\
1 & 1 & 0
\end{array}
\right), \
\BI_6=
\left(
\begin{array}{lll}
0 & 0 & 0\\
1 & 0 & 0
\end{array}
\right),\\
&\BI_7=
\left(
\begin{array}{lll}
0 & 1 & 0\\
1 & 0 & 0
\end{array}
\right)
\end{align*}
and we will denote the collections of such edges $E_1,E_2,\cdots, E_7$. 

Now the rate that a new edge is added to $E_t^{D,n}$ can be written as the follows:
\beq
\label{rate_of_discrepancy}
\begin{aligned}
\lambda^D(\hat A^{n}_t,\hat A^{n+1}_t)&=\sum_{\vec e\in E_1} \left|\harm_{\hat V^{n}_t}(\vec e)-\harm_{\hat V^{n+1}_t}(\vec e)\right|\\
&+\sum_{\vec e\in E_2}\harm_{\hat V^{n+1}_t}(\vec e)+\sum_{\vec e\in E_3}\harm_{\hat V^{n}_t}(\vec e)+\sum_{\vec e\in E_4}\harm_{\hat V^{n}_t}(\vec e)\\
&+\sum_{\vec e\in E_5}\harm_{\hat V^{n}_t}(\vec e)+\sum_{\vec e\in E_6}\harm_{\hat V^{n+1}_t}(\vec e)+\sum_{\vec e\in E_7}\harm_{\hat V^{n+1}_t}(\vec e). 
\end{aligned}
\eeq
For any $\vec e\in \cup_{i=2}^7 E_i$, note that at least one end point of $\vec e$ has to be within $\hat V^{n}_t\triangle \hat V^{n+1}_t\subset V_t^{D,n}$. Moreover, recall that for each point in $\BH$, there can be no more than 4 directed edges emanating from it and 4 edges going towards it. Thus, $| \cup_{i=2}^7 E_i|\le 8|V_t^{D,n}|\le 8|E_t^{D,n}|$. Now recalling $t<\Gamma$, $\hat A^{n}_t\cup \hat A^{n+1}_t\subset  [-n-\log n, n+\log n]\times[0,\log n]$, which implies that for each $\vec e\in \cup_{i=2}^7 E_i$, the corresponding harmonic measure in \eqref{rate_of_discrepancy} is bounded from above by $2\sqrt{\log n}$. Thus
\beq
\label{rate_of_discrepancy_1}
\begin{aligned}
&\sum_{\vec e\in E_2}\harm_{\hat V^{n+1}_t}(\vec e)+\sum_{\vec e\in E_3}\harm_{\hat V^{n}_t}(\vec e)+\sum_{\vec e\in E_4}\harm_{\hat V^{n}_t}(\vec e)\\
+&\sum_{\vec e\in E_5}\harm_{\hat V^{n}_t}(\vec e)+\sum_{\vec e\in E_6}\harm_{\hat V^{n+1}_t}(\vec e)+\sum_{\vec e\in E_7}\harm_{\hat V^{n+1}_t}(\vec e)\le  16|E_t^{D,n}|\sqrt{\log n}. 
\end{aligned}
\eeq
Now for each $\vec e=x\to y\in E_1$, by definition $x$ has to be in the inner boundary of $\hat V^{n}_t\cap \hat V^{n+1}_t$, while $y$ is in the complement of $\hat V^{n}_t\cup \hat V^{n+1}_t$. Moreover, we have 
\beq
\label{rate_of_discrepancy_2}
\left|\harm_{\hat V^{n}_t}(\vec e)-\harm_{\hat V^{n+1}_t}(\vec e)\right|\le \harm_{\hat V^{n}_t\cap \hat V^{n+1}_t}(\vec e)-\harm_{\hat V^{n}_t\cup \hat V^{n+1}_t}(\vec e).
\eeq
Using a similar method as in Section 5 of \cite{Procaccia_upper_2019} and recalling the definition of stationary harmonic measure, 
\begin{align*}
&\harm_{\hat V^{n}_t\cap \hat V^{n+1}_t}(\vec e)-\harm_{\hat V^{n}_t\cup \hat V^{n+1}_t}(\vec e)\\
=&\lim_{N\to\infty} \left(\harm_{\hat V^{n}_t\cap \hat V^{n+1}_t,N}(\vec e)-\harm_{\hat V^{n}_t\cup \hat V^{n+1}_t,N}(\vec e)\right)\\
=&\lim_{N\to\infty} \sum_{w\in L_N} \prob_w\left(X_{\tau_{(\hat V^{n}_t\cap \hat V^{n+1}_t)\cup L_0}}=x, \ X_{\tau_{(\hat V^{n}_t\cap \hat V^{n+1}_t)\cup L_0}-1}=y\right)\\
- &\lim_{N\to\infty} \sum_{w\in L_N}\prob_w\left(X_{\tau_{(\hat V^{n}_t\cup \hat V^{n+1}_t)\cup L_0}}=x, \ X_{\tau_{(\hat V^{n}_t\cup \hat V^{n+1}_t)\cup L_0}-1}=y\right)\\
=&\lim_{N\to\infty} \sum_{w\in L_N}\prob_w\left(X_{\tau_{(\hat V^{n}_t\cap \hat V^{n+1}_t)\cup L_0}}=x, \ X_{\tau_{(\hat V^{n}_t\cap \hat V^{n+1}_t)\cup L_0}-1}=y, \ X_{\tau_{(\hat V^{n}_t\cup \hat V^{n+1}_t)\cup L_0}}\in \hat V^{n}_t\triangle \hat V^{n+1}_t\right)\\
=& \lim_{N\to\infty} \sum_{w\in L_N}\sum_{z\in \hat V^{n}_t\triangle \hat V^{n+1}_t}\prob_w\left(X_{\tau_{(\hat V^{n}_t\cup \hat V^{n+1}_t)\cup L_0}}= z\right)\prob_z\left(X_{\tau_{(\hat V^{n}_t\cap \hat V^{n+1}_t)\cup L_0}}=x, \ X_{\tau_{(\hat V^{n}_t\cap \hat V^{n+1}_t)\cup L_0}-1}=y\right).
\end{align*}
Taking the summation over all $\vec e\in E_1$, and note that for all $z\in \hat V^{n}_t\triangle \hat V^{n+1}_t$, 
$$
\sum_{\vec e=x\to y\in E_1} \prob_z\left(X_{\tau_{(\hat V^{n}_t\cap \hat V^{n+1}_t)\cup L_0}}=x, \ X_{\tau_{(\hat V^{n}_t\cap \hat V^{n+1}_t)\cup L_0}-1}=y\right)\le 1
$$
since the summation above are over disjoint events. We have
$$
\sum_{\vec e\in E_1}\harm_{\hat V^{n}_t\cap \hat V^{n+1}_t}(\vec e)-\harm_{\hat V^{n}_t\cup \hat V^{n+1}_t}(\vec e)\le \harm_{\hat V^{n}_t\cup \hat V^{n+1}_t}(\hat V^{n}_t\triangle \hat V^{n+1}_t).
$$
Moreover, noting that by definition $\hat V^{n}_t\cup \hat V^{n+1}_t$ is connected in $\BH$, and that 
$$
|\hat V^{n}_t\triangle \hat V^{n+1}_t|\le |V^{D,n}_t|\le |E_t^{D,n}|,
$$
one may, by Theorem \ref{theorem: uniform_path} have,
\beq
\label{rate_of_discrepancy_3}
\sum_{\vec e\in E_1}\harm_{\hat V^{n}_t\cap \hat V^{n+1}_t}(\vec e)-\harm_{\hat V^{n}_t\cup \hat V^{n+1}_t}(\vec e)\le |E_t^{D,n}| \sqrt{\log n}. 
\eeq
Now combining \eqref{rate_of_discrepancy_1}-\eqref{rate_of_discrepancy_3} and plugging them back to \eqref{rate_of_discrepancy} gives us 
\beq
\label{rate_of_discrepancy_4}
\begin{aligned}
	\lambda^D(\hat A^{n}_t,\hat A^{n+1}_t)\le 17|E_t^{D,n}| \sqrt{\log n}
\end{aligned}
\eeq
Then recalling the definition of $\Delta_i$, by Poisson thinning and strong Markov property again we have 
$$
\prob\left(|E^{D,n}_1|\ge n^\alpha \right)=P\left(\sum_{i=1}^{n^\alpha}\Delta_i\le 1\right)\le P\left(\sum_{i=1}^{n^\alpha}\sigma_i\le 1\right)
$$
where $\{\sigma_i\}_{i=1}^{n^\alpha}$ is an independent sequence of exponential random variables with $\tilde \lambda_i=17i\sqrt{\log n}$.  

Thus, in order to prove Lemma \ref{lemma_number_discrepancies}, it suffices to prove the following result:

\begin{lemma}
	
	Let $\sigma_i$ be defined as above. Then for all $\alpha<1$, $\beta<\alpha$ and any  $c_3>0$, for all $n$ large enough
	
	$$
	\prob\left(\sum_{i=1}^{n^\alpha}\sigma_i<1\right)\le e^{-c_3 n^\beta}
	$$
	
\end{lemma}

\begin{proof}
	
	For $\beta<\alpha$ defined in the lemma and some constants $c_1,c_2>0$ (to be chosen later) define the events for $j\in[1, n^\alpha/n^\beta]\cap\BN$,
	
	$$G_j=\left\{\left|\left\{(j-1)n^\beta \le i\le j n^\beta: \sigma_i\ge \frac{c_2}{i\sqrt{\epsilon \log n}}\right\}\right|>c_1 n^\beta\right\}.$$

	Define $N_i=\ind_{\left\{\sigma_i\ge \frac{c_2}{i\sqrt{\epsilon \log n}}\right\}}$, thus $M_j=\sum_{i=(j-1)n^\beta}^{jn^\beta} N_i$ is a binomial random variable with parameters $n^\beta$ and $p=\prob\left(\sigma_i\ge \frac{c_2}{i\sqrt{\epsilon \log n}}\right)=e^{-c_2}$, which converges to $1$ when $c_2\rightarrow 0$. By the large deviation principle for binomial for binomial random variable 
	
	$$\prob(G_j^c)=\prob\left(M_j\le c_1 n^\beta\right) \le e^{-I(c_1,p)n^\beta}\le e^{-c_3 n^\beta},$$ where the last inequality follows by taking $p$ close enough to $1$ such that $I(c_1,p)>c'_3$ (see \cite{MR2571413} for the exact rate function). Since $c'_3$ was arbitrary, for a slightly smaller $c_3$ we can obtain for large enough $n$
	
	$$
	\prob\left( \bigcup_{j\in[1,\ldots, n^\alpha/n^\beta]\cap\BN}G_j\right)\le n^{\alpha-\beta}e^{-c'_3 n^\beta}\le e^{-c_3 n^\beta}
	.$$
	
	But under the event $\left\{ \bigcup_{j\in[1,\ldots, n^\alpha/n^\beta]\cap\BN}G_j\right\}^c$

	\begin{align*}
	&\sum_{i=1}^{n^\alpha}\sigma_i=\sum_{j=1}^{n^{\alpha-\beta}}\sum_{(j-1)n^\beta}^{j n^\beta}\sigma_i\ge \frac{c_2}{\sqrt{\epsilon\log n}}\left(\frac{c_1 n^\beta}{n^\beta}+\frac{c_1n^\beta}{2n^\beta}+\cdots +\frac{c_1 n^\beta}{n^{\alpha-\beta} n^\beta}\right)\\
	&> \frac{1}{2\epsilon} c_1c_2(\alpha-\beta)\sqrt{\log n}>1
	,\end{align*}
	
	where the last two inequalities require taking a large enough $n$.\end{proof}
Thus the proof of Lemma \ref{lemma_number_discrepancies} completes. \end{proof}

\section{Locations of discrepancies and proof of Theorem \ref{thm_summable}}

In the previous section, we have shown that, for any $\alpha>0$, by time $1$ with stretch exponentially high probability, there will be no more than $n^\alpha$ discrepancies. Now we show that it is highly unlikely that the first $n^\alpha$ possible discrepancies may ever reach our finite subset $K$. 

To show this, note that now the truncated model $(\hat A^n_t, \hat A^{n+1}_t)$ forms a finite state Markov process. In this section, it is more convenient to concentrate on the {\bf embedded chain}
$$
(\hat A^n_k, \hat A^{n+1}_k), \ k=0,1,2,\cdots 
$$ 
where all configuration $(\hat A^n_k, \hat A^{n+1}_k)$ with 
$$
\hat V^n_k\cup\hat V^{n+1}_k\nsubseteq [-n-\log n, n+\log n]\times[0,\log n]
$$
are absorbing states. 
\begin{remark}
	Without causing further confusion, we will, in this section use the parallel notations such as $(\hat A^n_k, \hat A^{n+1}_k)$, $V^{D,n}_k$ and $E_k^{D,n}$ etc., for the embedded chain without more specification. 
\end{remark}
Thus, in order to show Step (III), we only need to prove the lemma as follows:
\begin{lemma}
	\label{lemma_discrepancy_loc}
	There exists an $\alpha>0$ whose value will be specified later such that for any compact $K\subset \BH$, 
	$$
	\prob\left(E_{\Delta_{n^\alpha}}^{D,n}\cap K \not=\emptyset\right)\le n^{-1-\alpha}
	$$
	for all sufficiently large $n$. 
\end{lemma}

\begin{proof}
Now we recall the stopping times for the creation of new discrepancies:
$$
\Delta_i=\inf\{k\ge 0, \ |E^{D,n}_k|=i\},
$$
with the convention $\inf\emptyset=\infty$. We also define 
$$
\vec e_i=\left\{
\begin{aligned}
&E_{\Delta_i}^{D,n}\setminus E_{\Delta_{i-1}}^{D,n}, \text{ if } \Delta_i<\infty\\
&\emptyset \hspace{0.97 in} \text{otherwise}
\end{aligned}
\right..
$$


Noting that $\vec e_i$ is either empty of a singleton subset with one edge, we will, without loss of generality not specify the difference between the subset and the possible $i$th edge discrepancy. 

Now we are ready to introduce classifications on discrepancies as follows: 
\begin{itemize}
	\item For any $i=1$, we say $\vec e_1$ is {\bf good} if either $\vec e_1=\emptyset$ or 
	$$
	d(\vec e_1, (n+1,0))<n^{1-5\alpha}. 
	$$
	Here $d(\cdot,\cdot)$ is defined as the minimum distance over all endpoints.
	\item For any $i\ge 1$, we say $\vec e_i$ is {\bf good} if either $\vec e_i=\emptyset$ or 
	$$
	d(\vec e_i, E_{\Delta_{i-1}}^{D,n})<n^{1-5\alpha}. 
	$$
	Otherwise, we will say $\vec e_i$ is {\bf bad}. 
	\item If an $\vec e_i$ is bad, we call it {\bf devastating} if and only if $\vec e_i$ intersects with $[-n^{1-3\alpha}, n^{1-3\alpha}]\times [0,\log n]$. 
\end{itemize}
Moreover, one can also define 
$$
\kappa=\inf\{i\ge 1, \ s.t. \ \vec e_i \text{ is bad}\}. 
$$
By definition, one may see that $E_{\Delta_{n^\alpha}}^{D,n}\cap K \not=\emptyset$ only if either of the following two events happens:
\begin{itemize}
	\item Event $A$: $\kappa<n^\alpha$, and $\vec e_\kappa$ is devastating. 
	\item Event $B$: $\kappa<n^\alpha$, $\vec e_\kappa$ is bad but not devastating, and there is at least one bad event within $\kappa+1,\kappa+2,\cdots, n^\alpha$. 
\end{itemize}
To see the above assertion, one can from the definition of $A$ and $B$ see that $(A\cup B)^c$ can also be written as the union of $C\cup D$, where the events are defined as follows:
\begin{itemize}
	\item Event $C$: $\vec e_i$ are good for all $i=1,2, \cdots, n^\alpha$.
	\item Event $D$:  $\kappa<n^\alpha$, $\vec e_\kappa$ is bad but not devastating, and there are no bad events within $\kappa+1,\kappa+2,\cdots, n^\alpha$. 
\end{itemize}
Moreover, for each $i$, we define 
$$
l^+_i=\min\left\{x_1>0: \ s.t. \ \exists x_2 \text{ with $x=(x_1,x_2)$ a vertex for some edge within $E_{\Delta_i}^{D,n}$}  \right\},
$$
and
$$
r^-_i=\max\left\{x_1<0: \ s.t. \ \exists x_2 \text{ with $x=(x_1,x_2)$ a vertex for some edge within $E_{\Delta_i}^{D,n}$}  \right\}.
$$
Thus under event $C$ or $D$, 
$$
l^+_i\ge  n^{1-3\alpha}- n^\alpha\times n^{1-5\alpha}\ge n^{1-3\alpha}/2
$$
and 
$$
r^-_i\le -n^{1-3\alpha}+ n^\alpha\times n^{1-5\alpha}\le -n^{1-3\alpha}/2,
$$
which implies no discrepancy may be within $[-n^{1-3\alpha}/2,n^{1-3\alpha}/2]\times [0,\log n]\supset K$ for all sufficiently large $n$. 

Thus, now we only need to find the desired upper bound for the probability of events $A$ and $B$. For any $k$, define event 
$$
G_k=\{\vec e_i \text{ is good for $i=1,\cdots, k-1$}\}.
$$
\subsection{Upper bounds on $\prob(A)$}
\label{sub_prob_A}

For event $A$, by definition and strong Markov property one has 
\beq
\label{prob_A}
\begin{aligned}
\prob(A)&=\sum_{k=1}^{n^\alpha} \prob\left(G_k, \ \vec e_k \text{ is devastating}\right)\\
&=\sum_{k=1}^{n^\alpha} \sum_{j=0}^\infty\sum_{(\bar A_0, \tilde A_0)}\prob\left(G_k, \ \Delta_{k-1}<\infty, \ \Delta_{k}-\Delta_{k-1}>j, \ (\hat A^n_{\Delta_{k-1}+j}, \hat A^{n+1}_{\Delta_{k-1}+j})=(\bar A_0, \tilde A_0)\right)\\
&\hspace{1.21 in}\mathbb{P}_{(\bar A_0, \tilde A_0)}\left(\Delta_1=1, \vec e_1 \text{ is devastating}\right),
\end{aligned}
\eeq
where $\mathbb{P}_{(\bar A_0, \tilde A_0)}$ stands for the distribution of the the truncated embedded process $(\hat A^n_k,\hat A^{n+1}_k)$ starting from initial condition $(\bar A_0, \tilde A_0)$. 


At the same time, with similar calculation we have for any $k=1,2,\cdots, n^\alpha$
\beq
\label{prob_base}
\begin{aligned}
&\prob(G_k,\Delta_k<\infty)=\\
&\sum_{j=0}^\infty\sum_{(\bar A_0, \tilde A_0)}\prob\left(G_k, \Delta_{k-1}<\infty, \Delta_{k}-\Delta_{k-1}>j, (\hat A^n_{\Delta_{k-1}+j}, \hat A^{n+1}_{\Delta_{k-1}+j})=(\bar A_0, \tilde A_0)\right)\\
&\hspace{0.94 in}\mathbb{P}_{(\bar A_0, \tilde A_0)}\left(\Delta_1=1\right)\le 1. 
\end{aligned}
\eeq
Note that for any configuration $(\bar A_0, \tilde A_0)$ such that 
$$
\prob\left(G_k, \ \Delta_{k-1}<\infty, \ \Delta_{k}-\Delta_{k-1}>j, \ (\hat A^n_{\Delta_{k-1}+j}, \hat A^{n+1}_{\Delta_{k-1}+j})=(\bar A_0, \tilde A_0)\right)\not=0, 
$$
one must have $|\bar E_0\triangle \tilde E_0|\le k-1$. Now recalling the transition dynamic of the embedded chain, one has for all feasible $(\bar A_0, \tilde A_0)$ such that $\bar V_0\cup \tilde V_0\subset  [-n-\log n, n+\log n]\times[0,\log n]$
$$
\mathbb{P}_{(\bar A_0, \tilde A_0)}\left(\Delta_1=1\right)=\frac{\lambda^D(\bar A_0,\tilde A_0)}{\lambda^T(\bar A_0,\tilde A_0)}
$$ 
where $\lambda^D(\cdot,\cdot)$ was defined in \eqref{rate_of_discrepancy} and 
$$
\lambda^T(\bar A_0,\tilde A_0)=\sum_{\vec e} \max\{\harm_{\bar V_0}(\vec e), \harm_{\tilde V_0}(\vec e)\}.
$$
Otherwise $\mathbb{P}_{(\bar A_0, \tilde A_0)}\left(\Delta_1=1\right)=0$. Now for 
$$
\mathbb{P}_{(\bar A_0, \tilde A_0)}\left(\Delta_1=1, \vec e_1 \text{ is devastating}\right)
$$ 
recall that in \eqref{rate_of_discrepancy} we have 
$$
\begin{aligned}
	\lambda^D(\bar A_0,\tilde A_0)&=\sum_{\vec e\in E_1} \left|\harm_{\bar V_0}(\vec e)-\harm_{\tilde V_0}(\vec e)\right|\\
	&+\sum_{\vec e\in E_2}\harm_{\tilde V_0}(\vec e)+\sum_{\vec e\in E_3}\harm_{\bar V_0}(\vec e)+\sum_{\vec e\in E_4}\harm_{\bar V_0}(\vec e)\\
	&+\sum_{\vec e\in E_5}\harm_{\bar V_0}(\vec e)+\sum_{\vec e\in E_6}\harm_{\tilde V_0}(\vec e)+\sum_{\vec e\in E_7}\harm_{\tilde V_0}(\vec e). 
\end{aligned}
$$
For any $\vec e\in \cup_{i=2}^7E_i$, recall that at least one of the endpoints of $\vec e$ has to be in $\bar V_0\Delta \tilde V_0$. Thus it is easy to see 
$$
d(\vec e, E_{\Delta_{k-1}}^{D,n})=0. 
$$
Combining this with the fact that for all feasible $(\bar A_0,\tilde A_0)$, $\bar E_0\triangle \tilde E_0\subset (-\infty,-n+2n^{1-4\alpha})\cup (n-2n^{1-4\alpha},\infty)\times [0,\log n]$, which is disjoint with $[-2n^{1-3\alpha},2n^{1-3\alpha}]\times[0,\log n]$, we have 
\beq
\label{devastating_1}
\mathbb{P}_{(\bar A_0, \tilde A_0)}\left(\Delta_1=1, \vec e_1 \text{ is devastating}\right)\le \frac{\sum_{\vec e=x\to y\in E_1, |x_1|\le 2n^{1-3\alpha}} \left|\harm_{\bar V_0}(\vec e)-\harm_{\tilde V_0}(\vec e)\right|}{\lambda^T(\bar A_0,\tilde A_0)}
\eeq
when $\bar V_0\cup \tilde V_0\subset  [-n-\log n, n+\log n]\times[0,\log n]$ and equals to 0 otherwise. Thus for any configuration $(\bar A_0, \tilde A_0)$ such that 
$$
\prob\left(G_k, \ \Delta_{k-1}<\infty, \ \Delta_{k}-\Delta_{k-1}>j, \ (\hat A^n_{\Delta_{k-1}+j}, \hat A^{n+1}_{\Delta_{k-1}+j})=(\bar A_0, \tilde A_0)\right)\not=0, 
$$
and that 
$$
\mathbb{P}_{(\bar A_0, \tilde A_0)}\left(\Delta_1=1, \vec e_1 \text{ is devastating}\right)\not=0, 
$$
we have 
\beq
\label{devastating_2}
\frac{
	\mathbb{P}_{(\bar A_0, \tilde A_0)}\left(\Delta_1=1, \vec e_1 \text{ is devastating}\right)}{\mathbb{P}_{(\bar A_0, \tilde A_0)}\left(\Delta_1=1\right)}\le \frac{\sum_{\vec e=x\to y\in E_1, |x_1|\le 2n^{1-3\alpha}} \left|\harm_{\bar V_0}(\vec e)-\harm_{\tilde V_0}(\vec e)\right|}{\lambda^D(\bar A_0,\tilde A_0)}. 
\eeq
Now for the numerator of \eqref{devastating_2}, again we have 
\beq
\label{devastating_3}
\begin{aligned}
&\sum_{\vec e=x\to y\in E_1, |x_1|\le 2n^{1-3\alpha}} \left|\harm_{\bar V_0}(\vec e)-\harm_{\tilde V_0}(\vec e)\right|\\
\le\hspace{-0.4 in} &\sum_{\vec e=x\to y\in E_1, |x_1|\le 2n^{1-3\alpha}} \left[\harm_{\bar V_0\cap \tilde V_0}(\vec e)-\harm_{\bar V_0\cup \tilde V_0}(\vec e)\right]\\
=\hspace{-0.4 in} &\sum_{\vec e=x\to y\in E_1, |x_1|\le 2n^{1-3\alpha}}\sum_{z\in \bar V_0\Delta \tilde V_0} \harm_{\bar V_0\cup \tilde V_0}(z) \prob_z
\left(X_{\tau_{(\bar V_0\cap \tilde V_0)\cup L_0}-1}=y,X_{\tau_{(\bar V_0\cap \tilde V_0)\cup L_0}}=x \right)\\
& \hspace{0.25 in}\le  \harm_{\bar V_0\cup \tilde V_0}(\bar V_0\triangle \tilde V_0) \sup_{z\in \bar V_0\triangle \tilde V_0}\prob_z\left(\tau_{Box}<\tau_{L_0} \right),
\end{aligned}
\eeq


where 
$$
Box=[- 2n^{1-3\alpha}, 2n^{1-3\alpha}]\times[0,\log n]. 
$$
At the same time, note that for any feasible configuration $(\bar A_0, \tilde A_0)$, 
$$
\bar V_0\triangle \tilde V_0\subset Box_0=[n-2n^{1-4\alpha},n+\log n]\cup [-n-\log n, -n+2n^{1-4\alpha}]\times [0,\log n]
$$
which implies that 
\beq
\label{devastating_4}
\sup_{z\in \bar V_0\triangle \tilde V_0}\prob_z\left(\tau_{Box}<\tau_{L_0} \right)\le \sup_{z\in Box_0}\prob_z\left(\tau_{Box}<\tau_{L_0} \right). 
\eeq
Moreover, for each edge $\vec e=z\to w$ such that $z\in \bar V_0\triangle \tilde V_0$ and $w\notin \bar V_0\cup \tilde V_0$, by definition it has to belong to $E_3\cup E_6$ and thus by \eqref{rate_of_discrepancy}
\beq
\label{devastating_5}
\lambda^D(\bar A_0,\tilde A_0)\ge \harm_{\bar V_0\cup \tilde V_0}(\bar V_0\triangle \tilde V_0). 
\eeq
Now combining \eqref{prob_A}-\eqref{devastating_5} we have 
\beq
\label{devastating_6}
\prob(A)\le n^\alpha \sup_{x\in Box_0}\prob_x\left(\tau_{Box}<\tau_{L_0} \right). 
\eeq
Now we prove the following lemma:
\begin{lemma}
\label{lem_hitting}
For all $\alpha<1/5$ and all sufficiently large $n$
$$
\sup_{x\in Box_0}\prob_x\left(\tau_{Box}<\tau_{L_0} \right)\le n^{-1-2.5\alpha}. 
$$	
\end{lemma}
\begin{proof}
The proof of Lemma \ref{lem_hitting} follows a similar argument as in \cite{procaccia2018stationary}. Note that for any $x\in Box_0$,
$$
\prob_x\left(\tau_{Box}<\tau_{L_0} \right)\le \sum_{y\in \partial^{in} Box} \prob_x(\tau_{y}<\tau_{L_0}). 
$$
Then let $V_n=n/2\times[0,\infty)$, $V_n^1=n/2\times[0,n^4)$, and $V_n^2=n/2\times(n^4,\infty)$. By a similar argument as in \cite{procaccia2018stationary} we have 
\beq
\label{hitting_0}
\prob_x\left(\tau_{V_n}<\tau_{L_0} \right)\le n^{-1+\alpha/5}
\eeq
while 
$$
\prob_x\left(\tau_{V_n}<\tau_{L_0}, \tau_{V_n}=\tau_{V_n^2}\right)\le \frac{1}{n^3}. 
$$
Thus by strong Markov property, 
\beq
\label{hitting_1}
\begin{aligned}
\prob_x(\tau_{y}<\tau_{L_0})&= \sum_{z\in V_n} \prob_x\left(\tau_{V_n}<\tau_{L_0}, \tau_{V_n}=\tau_{z}\right) \prob_z(\tau_{y}<\tau_{L_0})\\
&\le \frac{1}{n^3}+ \sum_{z\in V_n^1} \prob_x\left(\tau_{V_n}<\tau_{L_0}, \tau_{V_n}=\tau_{z}\right) \prob_z(\tau_{y}<\tau_{L_0}).
\end{aligned}
\eeq
Moreover, for each $z\in V_n^1$, by reversibility of random walk (\cite{harmonic_measure_1987}), we have 
\beq
\label{hitting_2}
\prob_z(\tau_{y}<\tau_{L_0})\le \prob_y(\tau_{z}<\tau_{L_0})\ev_z[\# \text{ of visits to $z$ in $[0,\tau_{L_0})$}].
\eeq
For the first term in \eqref{hitting_2}, the same argument for \eqref{hitting_0} implies that 
$$
\prob_y(\tau_{z}<\tau_{L_0})\le \prob_y(\tau_{V_n}<\tau_{L_0})\le n^{-1+\alpha/5}.
$$
While for the second term in \eqref{hitting_2}, by \cite{procaccia2018stationary} we have there is a constant $C<\infty$ independent to $n$ such that for all $z\in V_n^1$
$$
\ev_z[\# \text{ of visits to $z$ in $[0,\tau_{L_0})$}]\le C\log n.
$$
Thus we have
\beq
\label{hitting_3}
\prob_z(\tau_{y}<\tau_{L_0})\le C n^{-1+\alpha/5}\log n. 
\eeq
Combining \eqref{hitting_0}-\eqref{hitting_3}, we have for any $x\in Box_0$, $y\in \partial^{in}Box$, 
$$
\prob_x(\tau_{y}<\tau_{L_0})\le C n^{-2+2\alpha/5}\log n. 
$$ 
Finally, noting that $|\partial^{in}Box|\le 5n^{1-3\alpha}$, we have
$$
\sup_{x\in Box_0}\prob_x\left(\tau_{Box}<\tau_{L_0} \right)\le C n^{-2+2\alpha/5}\log n \cdot n^{1-3\alpha} \le n^{-1-2.5\alpha}
$$
for all sufficiently large $n$. 
\end{proof}
Combining \eqref{devastating_6} and Lemma \ref{lem_hitting}, we have
\beq
\label{prob_A_final}
\prob(A)\le n^\alpha \sup_{x\in Box_0}\prob_x\left(\tau_{Box}<\tau_{L_0} \right)\le n^{-1-1.5\alpha}. 
\eeq

\subsection{Upper bounds on $\prob(B)$}
\label{sub_prob_B}

Now we find the upper bound for $\prob(B)$. Recall that 
\begin{itemize}
	\item Event $B$: $\kappa<n^\alpha$, $\vec e_\kappa$ is bad but not devastating, and there is at least one bad event within $\kappa+1,\kappa+2,\cdots, n^\alpha$. 
\end{itemize}
For any $k\ge 1$ define event 
$$
B_k=\left\{\vec e_1,\cdots, \vec e_{k-1} \text{ are good, $\vec e_k$ is bad} \right\}.
$$
Then by Markov property, we have 
\beq
\label{prob_B_1}
\begin{aligned}
\prob(B)=\sum_{k=1}^{n^\alpha-1} \sum_{(\bar A_0, \tilde A_0)}&\prob\left(B_k, \vec e_k \text{ is not devastating, } (\hat A^n_{\Delta_k},\hat A^{n-1}_{\Delta_k})=(\bar A_0, \tilde A_0)\right)\left(\sum_{j=1}^{n^\alpha-k} \mathbb{P}_{(\bar A_0, \tilde A_0)}(B_j)\right). 
\end{aligned}
\eeq
Using the argument in Subsection \ref{sub_prob_A} we have for all $k+j\le n^\alpha$ and any feasible configuration $(\bar A_0, \tilde A_0)$ such that 
$$
\prob\left(B_k, \vec e_k \text{ is not devastating, } (\hat A^n_{\Delta_k},\hat A^{n-1}_{\Delta_k})=(\bar A_0, \tilde A_0)\right)\not=0
$$

and that $\mathbb{P}_{(\bar A_0, \tilde A_0)}(B_i)$ not always =0 for all $i\le n^{\alpha}-k$, we have
$$
\mathbb{P}_{(\bar A_0, \tilde A_0)}(B_j)\le \mathbb{P}_{(\bar A_0, \tilde A_0)}(G_j, \Delta_j<\infty) P_{(0,\log n)}\left(\tau_{U_n}<\tau_{L_0}\right)\le P_{(0,\log n)}\left(\tau_{U_n}<\tau_{L_0}\right)
$$ 
where $U_n=\{-n^{1-5\alpha}/2,n^{1-5\alpha}/2\}\times [0,\infty)$. Again from \cite{procaccia2018stationary}, we have 
\beq
\label{prob_B_2}
\prob_{(0,\log n)}\left(\tau_{U_n}<\tau_{L_0}\right)\le n^{-1+6\alpha}. 
\eeq 
Thus by \eqref{prob_B_1} and \eqref{prob_B_2}, 
\beq
\label{prob_B_3}
\prob(B)\le n^{-1+7\alpha} \left(\sum_{k=1}^{n^\alpha-1} \prob(B_k)\right). 
\eeq
Again using the same argument, we have for any $k\le n^\alpha-1$, 
$$
P(B_k)\le P(G_k, \Delta_k<\infty)\prob_{(0,\log n)}\left(\tau_{U_n}<\tau_{L_0}\right)\le n^{-1+6\alpha} 
$$
which implies that 
\beq
\label{prob_B_3}
\prob(B)\le n^{-2+14\alpha}.
\eeq
Letting $\alpha=1/16$, then Lemma \ref{lemma_discrepancy_loc} follows from Lemma \ref{lem_hitting} and \eqref{prob_B_3}. \end{proof}

\begin{proof}[Proof of Theorem \ref{thm_summable}]

At this point, Theorem \ref{thm_summable} follows from the combination of Lemma \ref{lemma_number_discrepancies} and Lemma \ref{lemma_discrepancy_loc}. 
\end{proof}








\section{Proof of Theorem \ref{theorem_well_define}: Existence of the SDLA}
Theorem \ref{theorem_well_define} follows immediately once we show that the limiting process obtained by Theorem \ref{thm_summable} has the desired property.
\begin{lemma}\label{lem:approxharmonicmeasure}
	Fix a finite set $K$, $t>0$ and some $\epsilon>0$. $\exists N$ finite a.s., such that for all $n>N$, for all $0\le s\le t$ and any $x\in K$, 
	
	\begin{equation}\label{eq:needtoshowharm}
	|\CH_{L_0\cup A_s^n}(x)-\CH_{L_0\cup A_s}(x)|<\epsilon.
	\end{equation}
	
\end{lemma}
\begin{proof}
	By \cite[Lemma 2.6]{procaccia2018stationary} and the sub-linear growth of the interface model proved in Theorem \ref{thm_log_growth} and the fact we constructed all $A^n_s$ to be subsets of the interface model, there exists some $m>0$ such that for every every $n\in\BN\cup\{\infty\}$ and $x\in K$
	
	\begin{equation}\label{eq:harmconcentrating}
	\left|\sum_{|x|<m^{1.1}}\prob^{(x,m)}\left(S_{\tau_{L_0\cup A^n_s}}=x\right) -\CH_{L_0\cup A^n_s}(x)\right|<\epsilon/2
	.\end{equation}
	
	Let $K'\subset\BH$ be a large finite subset such that
	$$
	2m^{1.1}\max_{|x|<m^{1.1}}\prob^{(x,m)}(\tau_{K'^c}<\tau_K)<\epsilon/2
	.$$ 
	
	By Theorem \ref{thm_summable} we know that there is some $N\in\BN$ large enough such that for every $n>N$, 
	
	$$A_s^n\cap K'=A_s^N\cap K'=A_s\cap K'.$$
	Thus
	$$
	\left|\sum_{|x|<m^{1.1}}\prob^{(x,m)}\left(S_{\tau_{L_0\cup A^n_s}}=x\right) -\sum_{|x|<m^{1.1}}\prob^{(x,m)}\left(S_{\tau_{L_0\cup A_s}}=x\right)\right|<\epsilon/2
	.$$
	Together with \eqref{eq:harmconcentrating} we obtain \eqref{eq:needtoshowharm}.
	
\end{proof}



It remains to prove that $\{A_s\}_{s\le t}$ is Markov with the correct stationary harmonic measure as the infinitesimal generator.

\begin{lemma}\label{lem:correct_rate}
	For every finite subset $K\subset\BH$ and any $t>0$, for any $s\in[0,t]$ and $x\in K$,
	$$
	\lim_{\Delta s\to 0}\frac{\prob\left( A_{s+\Delta s}(x)=1|A_s(x)=0, \{A_\xi\}_{\xi\le s}\right)}{\Delta s}=\harm_{L_0\cup A_s}(x) \text{ a.s.} 
	$$
\end{lemma}

\begin{proof}
	Let $\epsilon>0$ and $G_n$ be the event that for all $s\le t$ and for all $x\in K$, $A_s^n(x)=A_s(x)$ and in addition, 
	$$
	|\CH_{L_0\cup A_s^n}(x)-\CH_{L_0\cup A_s}(x)|<\epsilon.
	$$
	By Lemma \ref{lem:approxharmonicmeasure} and Theorem \ref{thm_summable}, $\lim_{n\rightarrow\infty}\prob(G_n^c)=0$.
	Now uniformly for all $s<t$ and $\Delta s$ small enough, there is an $n\in\BN$ such that
	\begin{align*}
	&\prob\left( A_{s+\Delta s}(x)=1|A_s(x)=0, \{A_\xi\}_{\xi\le s}\right)\\
	&\in \prob\left( A_{s+\Delta s}(x)=1|A_s(x)=0, \{A_\xi\}_{\xi\le s}, G_n\right)+(-\epsilon,\epsilon)\\
	&=\prob\left( A^n_{s+\Delta s}(x)=1|A^n_s(x)=0, \{A_\xi\}_{\xi\le s}, G_n\right)+(-\epsilon,\epsilon)\\
	&\in\prob\left( A^n_{s+\Delta s}(x)=1|A^n_s(x)=0, |\CH_{L_0\cup A_s^n}(x)-\CH_{L_0\cup A_s}(x)|<\epsilon, A_s \right)+(-2\epsilon,2\epsilon)\\
	&\in (1-e^{\Delta s(\harm_{L_0\cup A_s}(x)+\epsilon)},1-e^{\Delta s(\harm_{L_0\cup A_s}(x)-\epsilon)})+(-2\epsilon,2\epsilon)
	,\end{align*}
	where we use dominated convergence theorem for the first and second approximations. Now taking $\epsilon\rightarrow 0$ and then $\Delta s\rightarrow 0$ we obtain the result. 
\end{proof}

\begin{proof}[Proof of Theorem \ref{theorem_well_define}]
By Lemma \ref{lem:correct_rate} we obtain that the almost sure limit $\{A_s\}_{s\le t}:=\lim_{m\rightarrow\infty} \{A_s^m\}_{s\le t}$ obtained in Theorem \ref{thm_summable} is a SDLA. 
\end{proof}








\section{Proof of Theroem \ref{thm:ergodicity}: Ergodocity of the SDLA}







\begin{proof}
	By Lemma \ref{lem:correct_rate} and the fact that the stationary harmonic measure is (well...) stationary, we obtain that $A^\infty_t$ is stationary with respect to the translation $\lambda_n(A^\infty_t)=A^\infty_t+n$, for any $n\in\BZ$. It is enough then to prove that $A^\infty_t$ is strongly mixing. Let $t>0$ and $K_1, K_2$ be two finite subsets of $\BH$ of distance  $\max\{|x_1-x_2|:x_1\in K_1, x_2\in K_2\}>2n$ ($n$ will be chosen big enough). We now consider two copies of $A_t^n$ constructed according to Poisson thinning of the same interface model, ${A}_t^n(1)$ is centered around an arbitrary point $x_1\in K_1$ and ${A}_t^n(2)$ is centered around an arbitrary point $x_2\in K_2$. For $i\in\{1,2\}$ and configurations $\xi_i\in\{0,1\}^{K_i}$.
	Define the events:
	\begin{align}
	\mathfs{B}_i&=\{A_t^\infty\cap K_i=\xi_i\}\\
	\mathfs{C}_i&=\{A_t^n(i)\cap K_i=\xi_i\}\\
	\mathfs{D}_i&=\{\max_{x\in A_t^n(i)}|x-x_i|<n/2\}
	\end{align}
	Under the event $\mathfs{D}_1\cap\mathfs{D}_2$ the events $\mathfs{C}_1$ and $\mathfs{C}_2$ are independent. This follows from the independence of Poisson processes on non intersecting domains. Moreover we know by Theorem \ref{thm_log_growth} that
	
	$$
	\lim_{n\rightarrow\infty}\prob\left(  \mathfs{D}_1^c\cup\mathfs{D}_2^c  \right)=0
	,$$
	and by Theorem \ref{thm_summable} that 
	$$\lim_{n\rightarrow\infty}\prob\left(\mathfs{B}_1\setminus \mathfs{C}_1\cup \mathfs{B}_2\setminus \mathfs{C}_2\right)=0.$$
	Thus
	\begin{align}
	\lim_{n\rightarrow\infty}\prob(\mathfs{B}_1\cap\mathfs{B}_2)&=\lim_{n\rightarrow\infty}\prob(\mathfs{C}_1\cap\mathfs{C}_2|\mathfs{D}_1\cap\mathfs{D}_2)=\lim_{n\rightarrow\infty}\prob(\mathfs{C}_1|\mathfs{D}_1\cap\mathfs{D}_2)\cdot \prob(\mathfs{C}_2|\mathfs{D}_1\cap\mathfs{D}_2)\\
	&=\lim_{n\rightarrow\infty}\prob(\mathfs{B}_1)\cdot \prob(\mathfs{B}_2)=\prob(\mathfs{B}_1)\cdot \prob(\mathfs{B}_2),
	\end{align}
	where in the last equality we used stationarity and abused notations to clarify that the limit is actually a constant sequence. 
\end{proof}







\section*{Acknowledgments}
We would like to thank Noam Berger for many fruitful discussions. Part of this paper was written while the first two authors were visitors of Peking university.

\bibliography{career}
\bibliographystyle{plain}

\end{document}